\documentclass[reqno]{article}
\usepackage{amsmath,amsfonts,amsthm,amssymb}
\usepackage{enumerate}
\usepackage{caption}
\usepackage[all]{xy}
\usepackage{subcaption}
\usepackage{float}
\usepackage{graphicx}
\usepackage{setspace}
\usepackage[stable]{footmisc}

\makeatletter
\renewcommand*\@fnsymbol[1]{\the#1} 
\def\blfootnote{\gdef\@thefnmark{}\@footnotetext} 
\makeatother

\newtheorem{theorem}{Theorem}[section]
\newtheorem{lemma}[theorem]{Lemma}
\newtheorem{corollary}[theorem]{Corollary}

\theoremstyle{definition} \newtheorem{definition}[theorem]{Definition}
\theoremstyle{definition} 
\theoremstyle{definition} 
\theoremstyle{definition} \newtheorem{example}[theorem]{Example}
\theoremstyle{definition}

\newcommand{\Z}{\mathbb{Z}}

\begin{document}
\onehalfspacing
\title{Commutation Semigroups of Finite Metacyclic Groups\\
       with Trivial Centre}
\author{
Darien DeWolf
\thanks{
  Department of Mathematics, Statistics and Computer Science,
  St. Francis Xavier University,
  2323 Notre Dame Ave,
  Antigonish, NS B2G 2W5,
  ddewolf@stfx.ca}\,
and
Charles C. Edmunds
\thanks{
  Department of Mathematics and Statistics,
  Mount Saint Vincent University,
  166 Bedford Highway,
  Halifax, NS B3M 2J6,
  cedmunds6868@gmail.com}}
\date{}
\maketitle

\abstract{
We study the right and left commutation semigroups of finite metacyclic groups
with trivial centre.
These are presented
\[
G(m,n,k) = \left\langle {a,b;{a^m} = 1,{b^n} = 1,{a^b} = {a^k}} \right\rangle
\quad (m,n,k\in\Z^+)
\]
with $(m,k - 1) = 1$ and $n = in{d_m}(k),$ the smallest positive integer for
which ${k^n} = 1\,\pmod m,$ with the conjugate of $a$ by $b$ written
${a^b}( = {b^{ - 1}}ab).$
The \emph{right} and \emph{left commutation semigroups of} $G,$ denoted
${\rm P}(G)$ and $\Lambda (G),$ are the semigroups of mappings generated by
$\rho (g):G \to G$ and $\lambda (g):G \to G$ defined by
$(x)\rho (g) = [x,g]$ and $(x)\lambda (g) = [g,x],$ where the commutator of
$g$ and $h$ is defined as $[g,h] = {g^{ - 1}}{h^{ - 1}}gh.$
This paper builds on a previous study of commutation semigroups of dihedral
groups conducted by the authors with C. Levy.
Here we show that a similar approach can be applied to $G,$ a metacyclic group
with trivial centre.
We give a construction of ${\rm P}(G)$ and $\Lambda (G)$ as unions of
\emph{containers}, an idea presented in the previous paper on dihedral groups.
In the case that $\left\langle a \right\rangle$ is cyclic of order $p$ or
${p^2}$ or its index is prime, we show that both ${\rm P}(G)$ and
$\Lambda (G)$ are disjoint unions of maximal containers.
In these cases, we give an explicit representation of the elements of each
commutation semigroup as well as formulas for their exact orders.
Finally, we extend a result of J. Countryman to show that, for $G(m,n,k)$ with
$m$ prime, the condition
$\left| {{\rm P}(G)} \right| = \left| {\Lambda (G)} \right|$ is equivalent to
${\rm P}(G) = \Lambda (G).$}

\vspace{1pc}


\section{Introduction}
\label{sec:1}

N.D. Gupta introduced the commutation semigroups of a group in
\cite{gupta1966}.
Given a group $G,$ the \emph{right} and \emph{left commutation maps}
associated with an
element $g \in G$ are the maps $\rho (g),\lambda (g):G \to G$ defined by
\[
(x)\rho (g) = [x,g]\mbox{ and }(x)\lambda (g) = [g,x],
\]
where the commutator of $x$ and $y$ is denoted
$[x,y]( = {x^{ - 1}}{y^{ - 1}}xy).$
Letting ${\cal M}(G)$ denote the semigroup, under composition, of all maps
from $G$ to $G,$ we define the \emph{right} and
\emph{left commutation semigroups}, denoted
${\rm P}(G)$ and $\Lambda (G),$ as the subsemigroups of ${\cal M}(G)$
generated by the sets ${{\rm P}_1}(G) = \left\{ {\rho (g):g \in G} \right\}$
and ${\Lambda _1}(G) = \left\{ {\lambda (g):g \in G} \right\}.$
Note that if $G$ is abelian, both commutation semigroups are trivial; thus for
the remainder of this paper we will consider only the case where $G$ is
non-abelian.

It is interesting to note that when $G = S_3,$ the symmetric group on three
letters, $|\mathrm{P}(G)| = 6$ and $|\Lambda(G)| = 9.$
One might have thought these two semigroups would be equal, or at least
isomorphic.
Thus the problem, originally asked by B.H. Neumann (oral communication to
N.D. Gupta), was: for which groups are the left and right commutation
semigroups (i) equal, (ii) isomorphic, or (iii) of equal order?

Gupta \cite{gupta1966} solved the isomorphism problem for dihedral groups and
showed that,
for $G$ nilpotent of class 2, 3, or 4, one has ${\rm P}(G) = \Lambda (G),$
${\rm P}(G) \cong \Lambda (G),$ and
$\left| {{\rm P}(G)} \right| = \left| {\Lambda (G)} \right|,$ respectively.
He also gave an example of a group nilpotent of class 5 for which the
commutation semigroups are not isomorphic.
In this context, since $S_3$ is not nilpotent, it is not surprising that its
commutation semigroups are different.

Extending the work of Gupta \cite{gupta1966}, Countryman \cite{countryman1970}
studied the commutation
semigroups of non-abelian $pq$-groups: $pq$-groups are extensions of a cycle
of order $p$ by a cycle of order $q$ with both $p$ and $q$ prime.
Since dihedral groups and $pq$-groups are metacyclic groups, the authors felt
that the techniques of \cite{dewolf2012}, \cite{dewolf2013}, and
\cite{gupta1966} might extend to all metacyclic groups.
We have chosen to restrict our discussion to metacyclic groups with trivial
centre, where a number of fairly general results may be obtained.
We will say more later about the decision to make this restriction.
We continue, in the spirit of \cite{dewolf2012} and \cite{dewolf2013}, to view
the commutation semigroups
in terms of containers.
In \cite{dewolf2013}, we were able to give formulas for the orders of the
commutation semigroups of finite dihedral groups.
For metacyclic groups, even those with trivial centre, we find that the
situation is complex enough that such formulas are not likely obtainable.
We will give examples illustrating how, even with trivial centre, the
number-theoretic complexity of the parameter $m$ makes the analysis more
difficult.
Despite this, we maintain that the method of containers is a powerful tool
with which to study commutation semigroups of metacyclic groups in general.

In Section \ref{sec:2}, we will show that the finite metacyclic groups with
trivial centre have presentations
\[
G(m,n,k)=\left\langle {a,b;{a^m} = 1,{b^n} = 1,{a^b} = {a^k}} \right\rangle,
\]
where $m,$ $n,$ and $k$ are positive integers, $(m,k - 1) = 1,$ and
$n = in{d_m}(k),$ the smallest positive integer for which
${k^n} = 1\,\pmod m,$ where the conjugate of $a$ by $b$ is written
${a^b}( = {b^{ - 1}}ab).$
Each group $G(m,n,k)$ has $\left\langle a \right\rangle$ as a (cyclic) normal
subgroup of order $m$ and of index $n.$
For different parameters, these presentations do not insure that the groups
presented are non-isomorphic, but they do give exactly the finite metacyclic
groups with trivial centre which we are studying.
Thus these presentations are adequate for our purposes.
It should be noted that, in \cite{hempel2000}, C.E. Hempel has classified the
finite metacyclic groups up to isomorphism.
$G(m,2,m - 1)$ is the dihedral group of order $2m$ and has trivial centre
exactly when $m$ is odd.
Also, every $pq$-group can be presented as $G(p,q,k).$
Thus our results will apply to \cite{countryman1970} on $pq$-groups as well as
to \cite{dewolf2012}, \cite{dewolf2013}, and \cite{levy2009} on dihedral
groups.

In \cite{levy2009}, C. Levy obtained formulas for the orders of both left and
right commutation semigroups for the dihedral groups $G(m,2,m - 1)$ with $m$
odd.
In \cite{dewolf2012}, D. DeWolf gave formulas for $G(m,2,m - 1)$ with $m$
even, and in \cite{dewolf2013},
formulas were produced which covered both cases.
For $G(m,2,m - 1)$ with $m$ odd the container structure is less complex than
when $m$ is even.
This is a consequence of the fact that when $m$ is odd, the dihedral group
$G(m,2,m - 1)$ has trivial centre.
As our work with metacyclic groups proceeded, we saw that the assumption of
trivial centre was a reasonable hypothesis to control some of the complexity.
Thus, from Section \ref{sec:3} onward, we will assume our groups have trivial
centre and can therefore be presented by some $G(m,n,k)$ as above.
This hypothesis is equivalent to requiring that $k - 1$ be coprime to $m,$ as
is shown in Section \ref{sec:2}, and will force the value of $m$ to be odd.
Note, however, that $G(9,3,4)$ has odd $m$ but also has trivial centre.
Analogues of many of our results hold for metacyclic groups with centre, but
we will leave them to a future study.

In Section \ref{sec:3}, we introduce \emph{mu-maps} and establish the
fundamental information we will need about \emph{containers}.

In Section \ref{sec:4}, we move to a more general setting which will include
both the left and the right commutation semigroups as particular cases of a
more general construction.
If $G$ has trivial centre, then, based on any set $S\subseteq\Z_m$
which contains both zero and
an invertible element, we will construct a semigroup ${\Sigma _G}(S),$ called
the $G$\emph{-semigroup based on} $S.$
Under certain hypotheses, this will be \emph{complete}, thereby allowing us to
give a full characterization of the mappings in ${\Sigma _G}(S)$ as well as a
formula for its exact order.
Applying this result to ${\rm P}(G)$ and $\Lambda(G)$ will give us an explicit
representation of the mappings they contain as well as formulas for their
orders.
This general approach may be of independent interest since it provides a
construction of many different semigroups of mappings from $G$ to $G.$

In Section \ref{sec:5}, we will discuss non-basic orbits, the one difficulty
that arises in the trivial centre case.
For ${\rm P}(G)$ and $\Lambda (G),$ it appears that this difficulty is fairly
rare.
We will show in Section \ref{sec:6} that all orbits of $G(m,n,k)$ are basic
when $m$ is prime or the square of a prime or when $n$ is prime.
A computer search has determined that, for ${\rm P}(G)$ and $\Lambda (G)$
with $G = G(m,n,k),$ the first non-basic orbits appear when
$m = 63 = {3^2} \cdot 7.$
Further searching gives the next problematic values of $m$ as
$75 = 3 \cdot {5^2},$ $81 = {3^4},$ $99 = {3^2} \cdot 11,$
$117 = {3^2} \cdot 13,$ and $125 = {5^3}.$
We conjecture that there are infinitely many of these cases.
The appearance of non-basic orbits appears to be correlated with the
complexity of the factorization of $m$ into primes.
Thus, in place of formulas, we will give a procedure which deals with
non-basic orbits and an example illustrating this procedure in action.
In principle, our methods will determine the commutation semigroups of any
metacyclic group with trivial centre, but our method is not uniform, depending
very much on the number theory of each individual group.

In Section \ref{sec:6}, we will give several applications of the general
theory applied to ${\rm P}(G)$ and $\Lambda (G).$
We will show that for $G(m,n,k)$ with trivial centre, if $m$ is prime or the
square of a prime, or if $n$ is prime, then ${\rm P}(G)$ and $\Lambda (G)$ are
\emph{complete} and can, therefore, be expressed as unions of maximal
containers.
Finally, we will re-state and extend the principle result of
\cite{countryman1970} showing that, for $G(m,n,k),$ if $n$ is prime, then
${\rm P}(G)$ and $\Lambda (G)$ are complete.

\section{Presenting finite metacyclic groups with trivial centre
\footnote{The authors express thanks to Prof. L.P. Comerford for his
          helpful comments on this section.}}
\label{sec:2}

In Lemma 2.1 of \cite{hempel2000}, C.E. Hempel gives a presentation,
which originated with H\"older, for finite metacyclic groups:
\begin{equation}\tag{$*$}
G = \left\langle {a,b;{a^m} = 1,{b^n} = {a^l},{a^b} = {a^k}} \right\rangle
\end{equation}
where $k,l,m,n\in\Z^+$ with ${k^n} = 1 \pmod m$ and $l(k - 1) = 0 \pmod m.$
These have $\left\langle a \right\rangle $ as a (cyclic) normal subgroup of
order $m$ and index $n.$

Since we will be studying finite metacyclic groups with trivial centre, we
will modify the presentation $(*)$ to produce a general presentation for all
finite metacyclic groups \emph{with trivial centre}.
The derivation of the presentation given in Corollary \ref{cor:2.3}
from $(*)$ is included since it is original and is not found in the
literature.
However, the details of this derivation can be skipped over without affecting
understanding of the rest of the paper.

We define the index of $k$ relative to $m,$ denoted $in{d_m}(k),$ to be the
smallest positive integer $d$ for which ${k^d} = 1 \pmod m.$
Note that this is the order of $k$ in the group of invertible elements of
$\Z_m.$
If $x$ is an element of a group $G,$ we denote its order by $ord(x).$
Recall that the conjugate of $a$ by $b$ is ${a^b} = {b^{ - 1}}ab,$ and the
commutator of $a$ and $b$ is $\left[ {a,b} \right] = {a^{ - 1}}{b^{ - 1}}ab.$
The commutator identities
$\left[ {xy,z} \right] = {\left[ {x,z} \right]^y}\left[ {y,z} \right]$ and
$\left[ {x,yz} \right] = \left[ {x,z} \right]{\left[ {x,y} \right]^z}$
will be used in this section and the next without further comment.

We begin with an elementary observation.

\begin{lemma}
\label{lemma:2.1}
For $k,m\in \Z^+,$ if ${k^n} = 1 \pmod m,$ then $(m,k) = 1.$
\end{lemma}
\begin{proof}
Suppose $(m,k) = g > 1,$ and $z\in\Z^+$ with ${k^n} = 1 + mz.$
Then ${k^n} = 1 + mz$ and, since $g$ divides ${k^n}$ and $m,$ $g$ divides $1$,
a contradiction.
\end{proof}

\begin{lemma}
\label{lemma:2.2}
For $k,l,m,n\in\Z^+$ with ${k^n} = 1 \pmod m$ and
$l(k - 1) = 0 \pmod m,$ the group
\[
G = \left\langle {a,b;{a^m} = 1,{b^n} = {a^l},{a^b} = {a^k}} \right\rangle
\]
has trivial centre if and only if $(m,k - 1) = 1$ and $n = in{d_m}(k).$
\end{lemma}
\begin{proof}
Suppose that there exists an $s\in\Z^+$ with $s < m$ and with $s$ least so
that ${a^s} = 1$ is a consequence of the relations given for $G.$
Since ${a^m} = 1,$ it follows that $s$ divides $m.$
We could apply Tietze transformations to the presentation to add the relation
${a^s} = 1$ and delete ${a^m} = 1.$
Note that when we replace $m$ by $s,$ the congruences, since $s$ divides $m,$
still hold.
We could then choose to replace the letter $s$ by $m$ throughout.
Thus we may say $ord(a) = m,$ without loss of generality.
It follows that we may assume $k,l < m.$
Note that
\[
\left[ {{a^i},b} \right]
  = {a^{ - i}}{({a^i})^b}
  = {a^{ - i}}{({a^i})^k} = {a^{i(k - 1)}}.
\]
Thus
\begin{enumerate}[(i)]
\item
${a^i} \in Z(G)$ if and only if $i(k - 1) = \;0 \pmod m.$
\item
Also since
\[
\left[ {a,{b^j}} \right]
  = {a^{ - 1}}{a^{{b^j}}}
  = {a^{ - 1}}{a^{{k^j}}}
  = {a^{{k^j} - 1}},
\]
it follows that ${b^j} \in Z(G)$ if and only if ${k^j} - 1 = \;0 \pmod m.$
\item
Letting $d = in{d_m}(k)$ we claim that ${b^d} \in Z(G).$
To see this, note that ${a^{{b^d}}} = {a^{{k^d}}} = a,$ since $d$ is the least
positive integer for which ${k^d} = 1 \pmod m.$
Thus ${b^d}$ commutes with both $a$ and $b$ and is, therefore, central in $G.$
\end{enumerate}
$\left(\Rightarrow\right)$
Assuming $G$ has trivial centre, we will first show that $(m,k - 1) = 1$ by
contradiction.
Suppose that $(m,k - 1) = g$ $(1 < g < m).$
Then there are positive integers $m'$ and $t$ such that $m = m'g,$
$k - 1 = tg,$ with $(m',t) = 1.$
Since $0 < m' < m$ and $ord(a) = m,$ we have ${a^{m'}} \ne 1,$ but
\[m'(k - 1) = m'tg = mt = 0 \pmod m.\]
Thus, by (i) above, we have $1 \ne {a^{m'}} \in Z(G) = \left\{ 1 \right\},$ a
contradiction.
From the relation ${b^n} = {a^l} \in \left\langle a \right\rangle ,$ we have
${a^{{b^n}}} = {a^{{a^l}}} = a;$ thus ${b^n}$ is central in $G$ and, therefore
by assumption, is trivial.
From this we see that ${a^l} = 1$ and, since $ord(a) = m,$ we have
$l = 0 \pmod m.$
Since ${b^n} = 1 \in \left\langle a \right\rangle ,$ we know there are
positive powers of $b$ in $\left\langle a \right\rangle .$ Suppose $j$ is the
least positive integer for which ${b^j} \in \left\langle a \right\rangle $ and
let $i$ be such that ${b^j} = {a^i}.$
Note that $a = {a^{{a^i}}} = {a^{{b^j}}} = {a^{{k^j}}};$ therefore,
${k^j} = 1 \pmod m,$ and since ${b^j}$ is central, ${b^j} = 1.$
Also note that, since $j$ was selected minimally, we have $j = in{d_m}(k).$
Dividing $n$ by $j,$ we have a positive integer $q$ and a non-negative integer
$r$ so that $n = qj + r\;(0 \le r < j),$ and hence
${a^l} = {b^n} = {({b^j})^q}{b^r} = {b^r}.$
This contradicts the minimality of $j$ unless $r = 0.$
Therefore $n = jq.$ Thus ${a^l} = {b^n} = {({b^j})^q} = {({a^i})^q},$ which
shows that ${b^j} = {a^i}$ implies ${b^n} = {a^l}.$
Since ${b^j} = {a^i}$ holds in $G,$ it is a consequence of the relations of
$G;$ thus, by Tietze transformations, we can add ${b^j} = {a^i}$ to the
relations of $G,$ and remove its consequence ${a^l} = {b^n}.$
As for the congruences on the parameters of the presentation, we have already
noted that $j = in{d_m}(k).$
Thus, as $j$ replaces $n$ in the relations when removing ${b^n} = {a^l}$ and
adding ${b^j} = {a^i},$ we drop the condition ${k^n} = 1 \pmod m$ and add
${k^j} = 1 \pmod m.$
Also $i$ replaces $l$ in the deleting of ${b^n} = {a^l}$ and adding
${b^j} = {a^i}.$
Thus we must see that $l(k - 1) = 0 \pmod m$ can be replaced by
$i(k - 1) = 0 \pmod m.$
This is the case because the relation ${b^j} = {a^i}$ implies that
${b^j} = 1,$ since it is central, and therefore ${a^i} = 1.$
This implies that $i = 0 \pmod m$ and therefore, $i(k - 1) = 0 \pmod m.$
Having applied these transformations, we may as well replace the letters $i$
and $j$ by $l$ and $n,$ respectively.

$\left(  \Leftarrow  \right)$
Suppose now that $(m,k - 1) = 1$ and $n = in{d_m}(k).$
Since $\left\langle a \right\rangle  \triangleleft G$ and
${b^n} = {a^l} \in \left\langle a \right\rangle ,$ the elements of the
quotient group
${\raise0.7ex\hbox{$G$} \!\mathord{\left/
 {\vphantom {G {\left\langle a \right\rangle }}}
 \right.\kern-\nulldelimiterspace}
\!\lower0.7ex\hbox{${\left\langle a \right\rangle }$}}$
are right cosets of $\left\langle a \right\rangle $, whose representatives are
powers of $b.$
$G$ is the union of these cosets; therefore, each element of $G$ can be
written in the form ${a^i}{b^j}$ $(0 \le i < m,$ $0 \le j < n).$
Suppose then that some ${a^i}{b^j} \in G$is central.
Note that
\[
1
= \left[ {b,{a^i}{b^j}} \right]
= \left[ {b,{b^j}} \right]{\left[ {b,{a^i}} \right]^{{b^j}}}
= {\left[ {b,{a^i}} \right]^{{b^j}}}
= {\left( {{{\left( {{a^{ - i}}} \right)}^b}{a^i}} \right)^{{b^j}}}
= {\left( {{{\left( {{a^{ - i}}} \right)}^k}{a^i}} \right)^{{k^j}}}
= {a^{i(k - 1){k^j}}}.
\]
Thus $i(k - 1){k^j} = 0 \pmod m.$ By Lemma \ref{lemma:2.1}, we know that
$k$ is invertible in and, by hypothesis, the same holds for $k - 1;$
therefore, we can reduce this congruence to $i = 0 \pmod m.$
It follows that for any ${a^i}{b^j} \in Z(G),$ we have
${a^i}{b^j} = {b^j} \in Z(G).$
From (ii) above, ${b^j}$ is central if and only if ${k^j} - 1 = 0 \pmod m.$
If $j < n,$ the statement ${k^j} - 1 = 0 \pmod m$ would contradict the
minimality of $n( = in{d_m}(k))$ unless $j = 0.$
Thus if ${a^i}{b^j}$ is central, it is trivial and, therefore,
$Z(G) = \left\{ 1 \right\},$ as required.
\end{proof}

\begin{corollary}
\label{cor:2.3}
For $m,n,k\in\Z^+,$ every finite metacyclic group with trivial centre can be
presented as
\[
G(m,n,k)
= \left\langle {a,b;{a^m} = 1,{b^n} = 1,{a^b} = {a^k}} \right\rangle
\]
where $(m,k - 1) = 1$ and $n = in{d_m}(k).$
\end{corollary}
\begin{proof}
We will begin with the presentation $(*),$
$G = \left\langle {a,b;{a^m} = 1,{b^n} = {a^l},{a^b} = {a^k}} \right\rangle,$
along with the conditions ${k^n} = 1 \pmod m$ and
$l(k - 1) = 0 \pmod m.$
We know that every finite metacyclic group has this presentation for some
$k,l,m,n\in\Z^+.$
Lemma \ref{lemma:2.2} says that the additional conditions, $(m,k - 1) = 1$ and
$n = in{d_m}(k),$ are necessary and sufficient to assure that the presentation
gives a finite metacyclic group with trivial centre.
Note that $n = in{d_m}(k)$ implies ${k^n} = 1 \pmod m;$ thus the latter can
be removed from the list as redundant.
The condition $(m,k - 1) = 1$ implies that $k - 1$ is invertible in $\Z_m.$
Multiplying both sides of the congruence $l(k - 1) = 0 \pmod m$ by the
inverse of $k - 1$ yields $l = 0 \pmod m.$
Therefore the relation $l = 0 \pmod m$ replaces $l(k - 1) = 0 \pmod m.$
Applying $l = 0 \pmod m$ to the only relation containing an $l$ replaces
${b^n} = {a^l}$ with ${b^n} = 1.$
Therefore, for $k,l,m,n\in\Z^+$ satisfying the conditions $(m,k - 1) = 1,$
$n = in{d_m}(k),$ and $l = 0 \pmod m,$ the presentations
$\left\langle {a,b;{a^m} = 1,{b^n} = 1,{a^b} = {a^k}} \right\rangle $ give
exactly the finite metacyclic groups with trivial centre.
Note that since the letter $l$ does not occur in the presentation, we may omit
the condition $l = 0 \pmod m$ without loss of generality.
\end{proof}

It will prove efficient to make the following notational conventions.
If $S$ is a subset of the multiplicative semigroup $\Z_m,$ we denote the
invertible elements of $S$ by $I(S)$ and the non-invertible elements of $S$
by $N(S).$
Recall that an element of is invertible if and only if it is coprime to $m.$
For each $t\;(0 \le t \le n)$ we let ${k_t} = {k^t} - 1 \pmod m.$
Thus ${k_1} = k - 1 \pmod m,$ ${k_0} = 0 \pmod m,$ and, since
${k^n} = 1 \pmod m,$ we have ${k_n} = 0 \pmod m.$

\begin{lemma}
\label{lemma:2.4}
If $G$ is a finite metacyclic group presented by $(*)$ (possibly having a
non-trivial centre) with $R = \{k_j\in\Z_m: j\in\Z_n\}$ and
$L = \{-k_j\in\Z_m: j\in\Z_n\}$ then
\begin{enumerate}[(i)]
\item
$0 \in R,$ $0 \in L,$ and
\item
if $n = in{d_m}(k),$ the following conditions are equivalent:
\begin{enumerate}[(a)]
\item the centre of $G$ is trivial,
\item $I(R) \ne \varnothing,$
\item $I(L) \ne \varnothing.$
\end{enumerate}
\end{enumerate}
\end{lemma}

\begin{proof}\hfill
\begin{enumerate}[(i)]
\item
Note that $0 = {k_0} \in R$ and $0 =  - {k_0} \in L.$
\item
$\left( {a \Rightarrow b} \right)$
Suppose first that the centre of $G$ is trivial and, hence, by
Lemma \ref{lemma:2.2}, we have $(m,{k_1}) = 1.$
It follows that ${k_1}( \in R)$ is invertible in $\Z_m.$

$\left( {b \Rightarrow c} \right)$
If, for some $j\in\Z_n,$ ${k_j} \in R$ is invertible in $\Z_m,$ then
$- {k_j} \in L.$
Denoting the inverse of ${k_j}$ in $\Z_m$ as $k_j^{-1},$ we see that
$( - {k_j})( - k_j^{ - 1}) = 1;$ therefore, $- {k_j}$ is invertible.
Hence $I(L) \ne \varnothing.$

$\left( {c \Rightarrow a} \right)$
Now suppose that there is a $j\in\Z_n$ for which $ - {k_j}$ is invertible.
Note that if $j = 0,$ then $ - {k_0} = 0\; \notin I(L).$
Therefore, we may assume that $0 < j < n.$
If $j = 1,$ we have ${k_1}$ invertible in $\Z_m$ and, hence, coprime to $m.$
Thus, along with the hypothesis $n = in{d_m}(k),$ Lemma \ref{lemma:2.2}
implies that $G$ has trivial centre.
If $1 < j < n,$ then
\[
- {k_j}
=  - {k_1}(1 + k +  \ldots  + {k^{j - 1}})
= {k_1}( - (1 + k +  \ldots  + {k^{j - 1}})).
\]
Since $ - {k_j}$ is invertible, so are both factors; therefore, ${k_1}$ is
invertible.
It follows, by Lemma \ref{lemma:2.2}, that $G$ has trivial centre.\qedhere
\end{enumerate}
\end{proof}

In Section \ref{sec:4}, we will use the sets $R$ and $L$ to construct right
and left commutation semigroups.
The previous lemma illustrates how the triviality of the centre of $G$
splits the number theory associated with the commutation semigroups into two
distinct cases: $R$ and $L$ will each contain $0,$ a non-invertible element,
but they will contain an invertible element exactly when the centre of $G$ is
trivial.
The existence of invertible elements in $R$ (and hence in $L)$ will allow us
to proceed with the arguments given below (see the definitions of
$G$\emph{-semigroup} and \emph{orbit}).
We will not give a complete description of the commutation semigroups in the
case that the centre of $G$ is trivial, but we will be able to obtain some
useful and rather general results with this assumption.
We will also show that our method will allow the calculation of the elements
of the commutation semigroups and their orders provided the reader is willing
to take on some cumbersome modular arithmetic calculations.
A theory for metacyclic groups with non-trivial centre could still be
approached using containers, but it would have to take a different form from
what we do below.

From this point onward, $G(m,n,k),$ abbreviated as $G,$ will be a finite
metacyclic group with trivial centre as described in Corollary \ref{cor:2.3}.

\section{Commutation mappings, mu-maps, and containers}
\label{sec:3}

We begin to study commutation mappings on $G$ with a general result about
commutators.

\begin{lemma}
\label{lemma:3.1}
If $G = G(m,n,k),$ $i,r\in\Z_m$ and $j,s\in\Z_n,$ then
$\left[ {{a^i}{b^j},{a^r}{b^s}} \right] = {a^N}$ where
$N = i{k^j}{k_s} - r{k^s}{k_j} \pmod m.$
\end{lemma}
\begin{proof}
\begin{align*}
\left[ {{a^i}{b^j},{a^r}{b^s}} \right]
&= {\left[ {{a^i},{a^r}{b^s}} \right]^{{b^j}}}
    \left[ {{b^j},{a^r}{b^s}} \right]
= {\left[ {{a^i},{b^s}} \right]^{{b^j}}}{
   \left[ {{b^j},{a^r}} \right]^{{b^s}}} \\
&= {\left( {{a^{ - i}}{{({a^i})}^{{b^s}}}} \right)^{{b^j}}}{
    \left( {{b^{ - j}}{a^{ - r}}{b^j}{a^r}} \right)^{{b^s}}}
= {\left( {{a^{ - i}}{{({a^i})}^{{k^s}}}} \right)^{{k^j}}}{
   \left( {{a^{ - r{b^j}}}{a^r}} \right)^{{k^s}}}
= {\left( {{a^{i({k^s} - 1)}}} \right)^{{k^j}}}{
   \left( {{a^{r(1 - {k^j})}}} \right)^{{k^s}}} \\
&= {a^{i({k^s} - 1){k^j} + r(1 - {k^j}){k^s}}}
= {a^{i{k^j}{k_s} + r{k^s}{k_j}}}.\qedhere
\end{align*}
\end{proof}

The following concept was introduced by N.D. Gupta in \cite{gupta1966}.

\begin{definition}
For $G = G(m,n,k),$ and $(x,y)$ a pair of elements, of a \emph{mu-map} is a
mapping $\mu (x,y):G \to G$ defined by
$\left( {{a^i}{b^j}} \right)\mu (x,y) = {a^N},$
where $N = xi{k^j} - y{k_j}\pmod m.$
\end{definition}

\begin{lemma}
\label{lemma:3.2}
For each $g \in G$ the mappings $\rho (g)$ and $\lambda (g)$ are mu-maps.
In particular if $g = {a^r}{b^s},$ then
$\rho ({a^r}{b^s}) = \mu ({k_s},r{k^s})$ and
$\lambda ({a^r}{b^s}) = \mu ( - {k_s}, - r{k^s}).$
\end{lemma}
\begin{proof}
Note that, by Lemma \ref{lemma:3.1},
\[
({a^i}{b^j})\rho ({a^r}{b^s})
= \left[ {{a^i}{b^j},{a^r}{b^s}} \right]
= {a^N},
\]
with $N = i{k^j}{k_s} - r{k^s}{k_j} \pmod m.$
By the definition of mu-map,
$({a^i}{b^j})\mu ({k_s},r{k^s}) = {a^{N'}}$ with
$N' = {k_s}i{k^j} - r{k^s}{k_j};$ thus
$\rho ({a^r}{b^s}) = \mu ({k_s},r{k^s}).$
Similarly
\[
({a^i}{b^j})\lambda ({a^r}{b^s})
= \left[ {{a^r}{b^s},{a^i}{b^j}} \right]
= {a^N},
\]
with $N = r{k^s}{k_j} - i{k^j}{k_s} \pmod m,$ while
$({a^i}{b^j})\mu ( - {k_s}, - r{k^s}) = {a^{N'}}$
with
$N' =  - {k_s}i{k^j} - ( - r{k^s}){k_j} =  - i{k^j}{k_s} + r{k^s}{k_j}.$
Therefore $\lambda ({a^r}{b^s}) = \mu ( - {k_s}, - r{k^s}).$
\end{proof}

The fundamental problem in constructing the commutation semigroups is that,
when taking products of rho-maps and lambda-maps, their products, in general,
are not rho-maps and lambda-maps.
Identifying the generating maps as mu-maps allows us a clearer view of how
these products are formed since products of mu-maps are mu-maps.

\begin{lemma}
\label{lemma:3.3}
If $\mu ({x_1},{y_1})$ and $\mu ({x_2},{y_2})$ are mu-maps, then their
composition is a mu-map with
\[
\mu ({x_1},{y_1}) \circ \mu ({x_2},{y_2}) = \mu ({x_1}{x_2},{y_1}{x_2}).
\]
\end{lemma}
\begin{proof}
\begin{align*}
({a^i}{b^j})\mu ({x_1},{y_1}) \circ \mu ({x_2},{y_2})
&= ({a^{{x_1}i{k^j} - {y_1}{k_j}}}{b^0})\mu ({x_2},{y_2})\\
&= {a^{{x_2}({x_1}i{k^j} - {y_1}{k_j}){k^0} - {y_2}{k_0}}}\\
&= {a^{{x_2}({x_1}i{k^j} - {y_1}{k_j})}} \\
&= {a^{{x_1}{x_2}i{k^j} - {y_1}{x_2}{k_j}}}\\
&= ({a^i}{b^j})\mu ({x_1}{x_2},{y_1}{x_2}).\qedhere
\end{align*}
\end{proof}

In light of this result we make the following definition.

\begin{definition}
The set $\mathrm{M}(G) = \{\mu(x,y): x,y\in\Z_m\}$ of all mu-maps forms a
semigroup under composition of mappings.
We will refer to ${\rm M}(G)$ as the $\mu$-\emph{semigroup associated with}
$G.$
\end{definition}

To obtain the commutation semigroups ${\rm P}(G)$ and $\Lambda (G),$
we will use Lemma \ref{lemma:3.2} to rewrite the generating sets
${{\rm P}_1}(G)$ and ${\Lambda _1}(G)$ as mu-maps and form their closures in
${\rm M}(G)$ under composition.
We can simplify this process further by grouping these mappings together into
sets called \emph{containers}.

\begin{definition}
For any pair $(x,y)\in\Z_m\times\Z_m,$ the
$(x,y)$\emph{-container with respect to} $G$ is the set
$C_G(x,y) = \{\mu(x,yz): z\in\Z_m\}.$
\end{definition}

When no confusion will arises, we abbreviate ${C_G}(x,y)$ as $C(x,y).$
We denote the order of the container by $\left| {C(x,y)} \right|.$
Note that by letting $z = 1$ in $\mu (x,yz)$ we see that
$\mu (x,y) \in C(x,y).$
Containers may intersect, but only in a limited way.

\begin{lemma}
\label{lemma:3.4}
For $G = G(m,n,k)$ and $x_1,x_2,y_1,y_2\in\Z_m,$
$C({x_1},{y_1}) \cap C({x_2},{y_2}) \ne \varnothing$
if and only if
${x_1} = {x_2}\pmod m.$
\end{lemma}
\begin{proof}
$\left(  \Rightarrow  \right)$
If $\mu \in C({x_1},{y_1}) \cap C({x_2},{y_2})$, then there exist
$z_1,z_2\in\Z_m$ such that
$\mu  = \mu ({x_1},{y_1}{z_1})  = \mu ({x_2},{y_2}{z_2}).$
Applying both maps to $a \in G,$ we have
$(a)\mu ({x_1},{y_1}{z_1}) = {a^{{N_1}}}$ with
${N_1} = {x_1} \cdot 1 \cdot {k^0} - {y_1}{z_1}{k_0} = {x_1}\pmod m,$
while $(a)\mu ({x_2},{y_2}{z_2}) = {a^{{N_2}}}$ with
${N_2} = {x_2} \cdot 1 \cdot {k^0} - {y_2}{z_2}{k_0} = {x_2}\pmod m.$
It follows that ${x_1} = {x_2}\pmod m.$

$\left(  \Leftarrow  \right)$
Note that $\mu ({x_1},0) = \mu ({x_1},{y_1} \cdot 0) \in C({x_1},{y_1})$ while
$\mu ({x_2},0) = \mu ({x_2},{y_2} \cdot 0)$$ \in C({x_2},{y_2}).$
But, since ${x_1} = {x_2}\pmod m,$ we have
$\mu ({x_1},0) = \mu ({x_2},0)$$ \in C({x_1},{y_1}) \cap C({x_2},{y_2}).$
Thus $C({x_1},{y_1}) \cap C({x_2},{y_2}) \ne \varnothing .$
\end{proof}

We need a preliminary lemma to calculate the orders of containers.

\begin{lemma}
\label{lemma:3.5}
Let $G = G(m,n,k)$ and $x,y\in\Z_m.$
Then, for all $z_1,z_2\in\Z_m,$
$\mu (x,y{z_1}) = \mu (x,y{z_2})$
if and only if
${z_1} = {z_2}\;\pmod {m'},$ where $m' = \frac{m}{{(m,y)}}.$
\end{lemma}
\begin{proof}
Letting $(m,y) = g,$ with $m = m'g$ and $y = y'g,$ it follows that
$(m',y') = 1.$
Notice that $\frac{m}{{(m,y)}} = \frac{{m'g}}{g} = m'.$

$\left(  \Rightarrow  \right)$
Supposing that $\mu (x,y{z_1}) = \mu (x,y{z_2}),$ we will apply both mappings
to $b \in G.$
This gives $(b)\mu (x,y{z_1}) = {a^{{N_1}}}$ with
${N_1} = x \cdot 0 \cdot {k^1} - y{z_1}{k_1}\;\;$ and
$(b)\mu (x,y{z_2}) = {a^{{N_2}}}$ with
${N_2} = x \cdot 0 \cdot {k^1} - y{z_2}{k_1}.$
It follows that $y{z_1}{k_1} = y{z_2}{k_1} \pmod m.$
One of the conditions on the presentation of $G$ is that $(m,{k_1}) = 1;$
therefore ${k_1}$ is invertible in $\Z_m$ and, multiplying both sides of the
congruence by $k_1^{ - 1},$ we have $y{z_1} = y{z_2} \pmod m.$
This can be rewritten $y'g{z_1} = y'g{z_2}\;\pmod {m'g}.$
Thus we have $y'{z_1} = y'{z_2}\;\pmod {m'}.$
Since $(m',y') = 1,$ $y'$ is invertible in $\Z_m.$
Thus we can multiply both sides of the congruence by the inverse of $y'$ in
$\Z_m$ to obtain ${z_1} = {z_2}\;\pmod {m'}.$

$\left(  \Leftarrow  \right)$
Conversely, we will assume that ${z_1} = {z_2}\pmod {m'}$ and show that when
the mappings $\mu (x,y{z_1})$ and $\mu (x,y{z_2})$ are applied to any
${a^i}{b^j} \in G$ the images are equal.
We begin with $({a^i}{b^j})\mu (x,y{z_1}) = {a^{{N_1}}}$ and
$({a^i}{b^j})\mu (x,y{z_2}) = {a^{{N_2}}}$ with
${N_1} = xi{k^j} - y{z_1}{k_j}$ and ${N_2} = xi{k^j} - y{z_2}{k_j}.$
Therefore, ${N_2} - {N_1} = y({z_1} - {z_2}){k_j} \pmod m.$
Our hypothesis is equivalent to ${z_1} - {z_2} = 0\;\pmod {m'}.$
Multiplying both sides of the congruence by $y'{k_j}$ yields
$y'({z_1} - {z_2}){k_j} = 0\;\pmod {m'}.$
This can then can be transformed to
$y'g({z_1} - {z_2}){k_j} = g \cdot 0\;\pmod {m'g},$ or
$y({z_1} - {z_2}){k_j} = 0 \pmod m.$
Thus ${N_2} - {N_1} = 0 \pmod m$ and our conclusion follows.
\end{proof}

\begin{corollary}
\label{cor:3.6}
If $G = G(m,n,k)$ and $x\in\Z_m$ then, for all $y_1,y_2\in\Z_m,$
$\mu (x,{y_1}) = \mu (x,{y_2})$ if and only if ${y_1} = {y_2} \pmod m.$
\end{corollary}
\begin{proof}
In Lemma \ref{lemma:3.5}, replace $y$ by $1,$ ${z_1}$ by ${y_1},$ and
${z_2}$ by ${y_2}.$
Note that $(m,y) = (m,1) = 1;$ thus $m' = m.$
\end{proof}

\begin{corollary}
\label{cor:3.7}
If $G = G(m,n,k)$ and $x,y\in\Z_m,$ then
$\left| {C(x,y)} \right| = \frac{m}{{(m,y)}}.$
\end{corollary}
\begin{proof}
From Lemma \ref{lemma:3.5}, there are exactly  $\frac{m}{{(m,y)}}$
distinct mappings in the container $C(x,y).$
\end{proof}

We will use the following lemmas in several of our examples.

\begin{lemma}
\label{lemma:3.8}
If $G = G(m,n,k),$ for each $x,y\in\Z_m,$
\begin{enumerate}[(i)]
\item $C(x,yz) \subseteq C(x,y),$
\item $C(x,y) \subseteq C(x,1),$
\item if $u\in I(\Z_m),$ then $C(x,y) = C(x,yu),$ and
\item $C(x,y) = C(x,1)$ if and only if $y\in I(\Z_m).$
\end{enumerate}
\end{lemma}
\begin{proof}\hfill
\begin{enumerate}[(i)]
\item
Let $\mu (x,(yz)w)$ be an arbitrary element of $C(x,yz)$ for some $w\in\Z_m.$
Since $wz\in\Z_m,$ we have $\mu (x,y(zw)) \in C(x,y)$ and our result follows.
\item
In part (i), let $y = 1$ and change $z$ to $y.$
\item
$\left(  \subseteq  \right)$
Let $\mu(x,y)$ $(z\in\Z_m)$ be an arbitrary element of $C(x,y).$
Since $zu^{-1}\in\Z_m,$ it follows that $\mu (x,yu(z{u^{ - 1}})) \in C(x,yu).$
But $\mu (x,yu(z{u^{ - 1}})) = \mu (x,yz).$
Thus we have shown that $\mu (x,yz) \in C(x,yu).$

$\left(  \supseteq  \right)$
This is immediate from part (i).
\item
$\left(  \Rightarrow  \right)$
Since $\mu (x,1) \in C(x,1) = C(x,y),$ there is a $z\in\Z_m$ so that
$\mu (x,yz) = \mu (x,1).$
By Corollary \ref{cor:3.6}, we have $yz = 1 \pmod m,$ from which it follows
that $y\in U(\Z_m).$

$\left(  \Leftarrow  \right)$
This follows directly from part (iii) by letting $y = 1.$\qedhere
\end{enumerate}
\end{proof}

\begin{lemma}
\label{lemma:3.9}
If $G = G(m,n,k)$ and $x,y_1,y_2\in\Z_m,$ then
$C(x,{y_1}) \subseteq C(x,{y_2})$ if and only if there exists $z\in\Z_m$ such
that ${y_1} = {y_2}z \pmod m.$
\end{lemma}
\begin{proof}
$\left(  \Rightarrow  \right)$
We have $\mu (x,{y_1}) \in C(x,{y_1}) \subseteq C(x,{y_2});$ therefore, there
exists $z\in\Z_m$ such that $\mu (x,{y_1}) = \mu (x,{y_2}z).$
Applying these mappings to $b,$ we obtain $(b)\mu (x,{y_1}) = {a^{{N_1}}}$
where ${N_1} =  - {y_1}{k_1} \pmod m$ and $(b)\mu (x,{y_2}z) = {a^{{N_2}}}$
where ${N_2} =  - {y_1}z{k_1} \pmod m.$
Thus ${y_1}{k_1} = {y_2}z{k_1} \pmod m$ and, since ${k_1}$is invertible,
we have ${y_1} = {y_2}z \pmod m.$

$\left(  \Leftarrow  \right)$
The fact that $C(x,{y_1}) = C(x,{y_2}z)  \subseteq C(x,{y_2})$ follows
immediately from Lemma \ref{lemma:3.8}(i).
\end{proof}

\section{A generalized approach}
\label{sec:4}

Recall that if $\varnothing\neq S \subseteq \Z_m,$ ${S^*}$ denotes the
subsemigroup of $Z_m$ generated by $S,$ and the invertibles $I({S^*})$ form a
subgroup of $\Z_m.$
It follows that $1 \in I({S^*})$ and, since $I({S^*})$ is a finite group, for
each $x \in I({S^*}),$ there is a least non-negative integer $u$ for which
${x^u} = 1.$
Thus ${x^{ - 1}} = {x^{u - 1}} \in I({S^*}).$

\begin{definition}
A non-empty subset S of $\Z_m$ is a \emph{base} if $0 \in S$ and $I(S)$ is
non-empty.
\end{definition}

\begin{definition}
For $G = G(m,n,k)$ and $S$ a base, the $G$\emph{-semigroup based on} $S$,
denoted ${\Sigma _G}(S),$ is the subsemigroup of ${\rm M}(G)$ generated by
$\Gamma_\mu(S) = \{\mu(s,z): s\in S, z\in\Z_m\}.$
We call the set ${\Gamma_\mu }(S)$ the set of $\mu$\emph{-generators
associated with} $S$ and the set
$\Pi_\mu(S) = \{\mu(ss^*, s^*z): s\in S, s^*\in S^*, z\in\Z_m\}$
the set of $\mu$\emph{-products associated with} $S.$
\end{definition}

\begin{lemma}
\label{lemma:4.1}
For $G = G(m,n,k)$ and $S$ a base,
${\Sigma _G}(S) = {\Gamma _\mu }(S) \cup {\Pi _\mu }(S).$
\end{lemma}
\begin{proof}
We first show that ${\Pi _\mu }(S)$ is the set of products of two or more
$\mu$-generators.
Suppose we form the product of two or more generators
$\mu ({s_1},{z_1})\mu ({s_2},{z_2}) \cdots \mu ({s_t},{z_t}).$
By repeated use of Lemma \ref{lemma:3.3}, the product can be written
$\mu ({s_1}{s_2} \cdots {s_t},{z_1}{s_2} \cdots {s_t}).$
Note then that ${s_1}$ could be any element of $S$ and ${s_2} \cdots {s_t}$
represents an arbitrary element of ${S^*};$
therefore, each $\mu (s{s^*},z{s^*}) \in {\Pi _\mu }(S)$ is such a product and
each product is an element of  ${\Pi _\mu }(S).$
Since we have included the generating set ${\Gamma _\mu }(S)$ and all products
of generators, it is clear that
${\Sigma _G}(S) = {\Gamma _\mu }(S) \cup {\Pi _\mu }(S).$
\end{proof}

By proper selection of $S,$ we will be able to produce both the left and right
commutation semigroups as particular instances of ${\Sigma _G}(S).$
The theorems we want to exhibit for ${\rm P}(G)$ and $\Lambda (G)$ will follow
immediately from the same results for ${\Sigma _G}(S).$
In addition to representing the commutation semigroups, the construction of
${\Sigma _G}(S)$ produces a subsemigroup of ${\rm M}(G)$ for each choice of a
base $S$ and therefore may be worthy of further study on its own.

First we will establish that the commutation semigroups are, indeed, instances
of ${\Sigma _G}(S).$

\begin{lemma}
\label{lemma:4.2}
If $G = G(m,n,k),$ $R = \{k_j\pmod m: j\in\Z_n\},$ and
$L = \{-k_j\pmod m: j\in\Z_n\}$ then $R$ and $L$ are bases with
${\rm P}(G) = {\Sigma _G}(R)$ and $\Lambda (G) = {\Sigma _G}(L).$
\end{lemma}
\begin{proof}
Before we can form ${\Sigma _G}(S),$ we must confirm that $S$ is a base;
in particular, we must show that $R$ and $L$ are bases.
By Lemma \ref{lemma:2.4}, we have zero in both $N(R)$ and $N(L),$ and since
$G$ has trivial centre, $I(R)$ and $I(L)$ are non-empty.
Therefore $R$ and $L$ are bases.
We will prove ${\rm P}(G) = {\Sigma _G}(R)$ and note that a similar argument
can be given to prove $\Lambda (G) = {\Sigma _G}(L).$
By Lemma \ref{lemma:3.2}, we have $\rho ({a^r}{b^s}) = \mu ({k_s},r{k^s})$
for each $r\in\Z_m, s\in\Z_n.$
Since ${k_s} \in R$ and $r{k^s} \in {\Z_m},$
$\mu ({k_s},r{k^s}) \in {\Gamma _\mu }(R).$
Since $k,$ and thus ${k^s},$ is invertible in $\Z_m,$ it follows that
$\{rk^s: r\in \Z_m\} = \Z_m.$
Every element of ${\Gamma _\mu }(R)$ occurs in the form $\mu ({k_s},r{k^s});$
therefore, $\{\rho(a^ib^j):i\in\Z_m, j\in\Z_n\} = \Gamma_\mu(R).$
Since ${\rm P}(G)$ and ${\Sigma _G}(R)$ are generated by the same mappings,
they are equal.
\end{proof}

\begin{lemma}
\label{lemma:4.3}
Suppose $G = G(m,n,k)$ and $S$ is a base.
For each $x,y\in\Z_m,$ $\mu (x,y) \in {\Sigma _G}(S)$ if and only if
$C(x,y) \subseteq {\Sigma _G}(S).$
\end{lemma}
\begin{proof}
$\left(  \Rightarrow  \right)$
Given any $z\in\Z_m$ we wish to show that if $\mu (x,y) \in {\Sigma _G}(S),$
then $\mu (x,yz) \in {\Sigma _G}(S).$
Suppose that $\mu (x,y) \in {\Gamma _\mu }(S);$ then $x \in S$ and $y\in\Z_m.$
Thus $yz\in\Z_m$ and clearly
$\mu (x,zy) \in {\Gamma _\mu }(S) \subseteq {\Sigma _G}(S).$
If $\mu (x,y) \in {\Pi _\mu }(S),$ then we know there are $s \in S,$
${s^*} \in {S^*},$ and $z'\in\Z_m$ so that $x = s{s^*}\pmod m$ and
$y = {s^*}z'\pmod m.$
Therefore, $yz = {s^*}z'z \pmod m.$
Thus
$\mu (x,yz) = \mu (s{s^*},{s^*}z'z)
  \in {\Pi _\mu }(S) \subseteq {\Sigma _G}(S).$
In each case, $\mu (x,yz) \in {\Sigma _G}(S).$

$\left(  \Leftarrow  \right)$
We know that $\mu (x,y) \in C(x,y).$
Thus, assuming $C(x,y) \subseteq {\Sigma _G}(S),$ it is immediate that
$\mu (x,y) \in {\Sigma _G}(S).$
\end{proof}

The following lemma shows that each $x$ in ${S^*}$ produces at least one
container in ${\Sigma _G}(S).$

\begin{lemma}
\label{lemma:4.4}
If $G = G(m,n,k)$ and $S$ is a base, then for each $x \in {S^*},$ there exists
$y \in {S^*}$ so that $C(x,y) \subseteq {\Sigma _G}(S).$
\end{lemma}
\begin{proof}
If $x \in {S^*},$ then $x \in S$ or $x$ is a product of elements of $S.$
If $x \in S$ then, selecting $y = 1,$ we obtain
$\mu (x,1) \in {\Gamma _\mu }(G) \subseteq {\Sigma _G}(S).$
Therefore, by Lemma \ref{lemma:4.3}, we have
$C(x,1) \subseteq {\Sigma _G}(S),$ as required.
If $x = {s_1}{s_2} \ldots {s_t}\;({s_i} \in S,t > 1),$ let
$x' = {s_2} \ldots {s_t}.$
Since $x' \in {S^*},$ we see that
$\mu ({s_1}x',x') \in {\Pi _\mu }(S) \subseteq {\Sigma _G}(S).$
Since $x = {s_1}x',$ if we select $y = x',$
Lemma \ref{lemma:4.3} implies that $C(x,y) \subseteq {\Sigma _G}(S).$
\end{proof}

We now introduce a set $Y(x)$ associated with each $x$ in ${S^*}.$
The following lemma characterizes exactly those $y$-values for which
$C(x,y) \subseteq {\Sigma _G}(S).$

\begin{lemma}
\label{lemma:4.5}
If $G = G(m,n,k),$ $S$ is a base, $x \in {S^*},$ and
\[
Y(x) =
\{
s^*z: s^*\in S^*, z\in\Z_m,
\exists s\in S \mbox{ so that } x = ss^*\pmod m
\},
\]
then $y \in Y(x)$ if and only if $C(x,y) \subseteq {\Sigma _G}(S).$
\end{lemma}
\begin{proof}
It will be convenient to suppress mention of the modulus $m.$

$\left(  \Rightarrow  \right)$
If $y \in Y(x)$ then, there exist ${s^*}\in {S^*},$ $z\in\Z_m,$ and
$s \in S$ with $x = s{s^*}$ and $y = {s^*}z.$
By Lemma \ref{lemma:4.4}, we know there exists $y' \in {S^*}$ such that
$\mu ({s^*},y') \in {\Sigma _G}(S).$
Also, we have $\mu (s,z) \in {\Gamma _\mu }(S) \subseteq {\Sigma _G}(S).$
Therefore, $\mu (s,z)\mu ({s^*},y') \in {\Sigma _G}(S).$
Note that $\mu (s,z)\mu ({s^*},y')$$ = \mu (s{s^*},{s^*}z) = \mu (x,y).$
And, since $\mu (x,y) \in {\Sigma _G}(S),$ Lemma \ref{lemma:4.3} implies that
$C(x,y) \subseteq {\Sigma _G}(S).$

$\left(  \Leftarrow  \right)$
Suppose now that $C(x,y) \subseteq {\Sigma _G}(S).$
By Lemma \ref{lemma:4.3}, we have $\mu (x,y) \in {\Sigma _G}(S).$
By Lemma \ref{lemma:4.1}, $\mu(x,y)\in {\Gamma_\mu }(S)$ or ${\Pi_\mu }(S).$
In the first case, we have $x \in S.$
Since $I({S^*})$ is a group, we know $1 \in I({S^*}) \subseteq {S^*};$
therefore, we let $s = x,$ ${s^*} = 1,$ and $z = y$ to obtain
$x = s{s^*} = x \cdot 1,$ with $y = {s^*}z = 1 \cdot y.$
Therefore, $y \in Y(x).$
In the second case, $\mu (x,y) \in {\Pi _\mu }(S).$
Thus we have $s \in S,{s^*} \in {S^*}$ and $z\in\Z_m,$ with $x = s{s^*}$
and $y = {s^*}z.$
It follows that $y \in Y(x).$
\end{proof}

Note that by Lemma \ref{lemma:4.4}, given any $x \in {S^*}$ there is a
$y \in {S^*}$ for which $C(x,y) \subseteq {\Sigma _G}(S);$
furthermore, Lemma \ref{lemma:4.5} determines exactly those values of $y$ for
which $C(x,y) \subseteq {\Sigma _G}(S).$
We will refer to these containers as a \emph{family}.

\begin{definition}
Suppose $G = G(m,n,k)$ and $S$ is a base.
For each $x \in {S^*},$ the $x$\emph{-family of containers (with respect to}
$G$ \emph{and} $S)$ is the set
${{\cal F}_G}(x,S) = \left\{ {C(x,y):y \in Y(x)} \right\}.$
We denote the union of the $x$-family by
$ \cup {{\cal F}_G}(x,S) = \bigcup\limits_{y \in Y(x)} {C(x,y)} .$
\end{definition}

\begin{theorem}
\label{thm:4.6}
If $G = G(m,n,k)$ and $S$ is a base, then
\[
\Sigma_G(S) = \dot{\bigcup}_{x\in S^*}\left(\cup {{\cal F}_G}(x,S)\right).
\]
\end{theorem}
\begin{proof}
Note that the union is disjoint by Lemma \ref{lemma:3.4}.

$\left(  \subseteq  \right)$
If $\mu ({x_0},{y_0}) \in {\Sigma _G}(S),$ then, by Lemma \ref{lemma:4.3},
$C({x_0},{y_0}) \subseteq {\Sigma _G}(S).$
Thus, by Lemma \ref{lemma:4.5}, ${y_0} \in Y({x_0}).$
Therefore $C({x_0},{y_0}) \in {{\cal F}_G}({x_0},S),$ which implies that
$\mu ({x_0},{y_0}) \in C({x_0},{y_0}) \subseteq  \cup {{\cal F}_G}({x_0},S).$
Therefore
\[
\mu(x_0, y_0) \in \dot{\bigcup}_{x\in S^*}\left(\cup {{\cal F}_G}(x,S)\right).
\]
$\left(  \supseteq  \right)$
If
\[
\mu(x_0, y_0) \in \dot{\bigcup}_{x\in S^*}\left(\cup {{\cal F}_G}(x,S)\right),
\]
it follows, by Lemma \ref{lemma:3.4}, that
\[
\mu ({x_0},{y_0}) \in  \cup {{\cal F}_G}({x_0},S)
= \bigcup\limits_{y \in Y({x_0})} {C({x_0},y)} .
\]
Therefore ${y_0} \in Y({x_0})$ and thus, by Lemma \ref{lemma:4.5},
$C({x_0},{y_0}) \subseteq {\Sigma _G}(S).$
By Lemma \ref{lemma:4.3}, $\mu ({x_0},{y_0}) \in {\Sigma _G}(S).$
\end{proof}

Theorem \ref{thm:4.6} states that ${\Sigma _G}(S)$ is the disjoint union of
all the $x$-families.
Since distinct families are disjoint, the complexity involved in representing
${\Sigma _G}(S)$ as a union of containers occurs entirely within each
$x$-family.
In this section, we will determine conditions that assure a minimal amount of
complexity, so that this union is easily determined.
In Section \ref{sec:5}, we will study the more involved situation.

\begin{definition}
If $G = G(m,n,k)$ and $S$ is a base, for each $x \in {S^*},$ we say the
$x$-family ${\cal F}(x,S)$ is \emph{complete} if $C(x,1) \in {\cal F}(x,S).$
The $G$-semigroup ${\Sigma _G}(S)$ is \emph{complete} if each
$x$-family is complete .
\end{definition}

Not all $x$-families are complete.
In Section \ref{sec:5}, Example 5.1 will show, for $G = G(63,6,2),$
that ${\cal F}(21,\left\{ {0,1,3,7,15,31} \right\})$ is not complete.

\begin{lemma}
\label{lemma:4.7}
For $x \in {S^*},$ if ${\cal F}(x,S)$ is complete, then
$ \cup {\cal F}(x,S) = C(x,1).$
\end{lemma}
\begin{proof}
$\left(  \subseteq  \right)$
By Lemma \ref{lemma:3.8}(ii), we have $C(x,y) \subseteq C(x,1)$ for each
$y\in\Z_m.$
Therefore,
\[
\cup {{\cal F}_G}(x,S) = \bigcup\limits_{y\in Y(x)} {C(x,y)} \subseteq C(x,1).
\]

$\left(  \supseteq  \right)$
Since ${\cal F}(x,S)$ is complete, we know that $C(x,1) \in {\cal F}(x,S);$
therefore, $C(x,1) \subseteq  \cup {\cal F}(x,S).$
\end{proof}

\begin{theorem}
\label{thm:4.8}
If $G = G(m,n,k),$ $S$ is a base, and ${\Sigma _G}(S)$ is complete, then
$\displaystyle\Sigma_G(S) = \dot{\bigcup}_{x\in S^*} C(x,1)$
and $\left| {{\Sigma _G}(S)} \right| = m\left| {{S^*}} \right|.$
\end{theorem}
\begin{proof}
By Theorem \ref{thm:4.6},
\[
\Sigma_G(S) = \dot{\bigcup}_{x\in S^*}\left(\cup {{\cal F}_G}(x,S)\right).
\]
Since each $x$-family is complete, Lemma \ref{lemma:4.7} implies that
\[
\Sigma_G(S) = \dot{\bigcup}_{x\in S^*} C(x,1).
\]
By Corollary \ref{cor:3.7},
$\left| {C(x,1)} \right| = \frac{m}{{(m,1)}} = m.$
Therefore $\left| {{\Sigma _G}(S)} \right| = m\left| {{S^*}} \right|.$
\end{proof}

Thus, if ${\Sigma _G}(S)$ is complete, we have the simplest situation.
${\Sigma _G}(S)$ is a disjoint union of maximal containers and its order is
easily calculated.
At this point we turn our attention to incomplete $x$-families.

\begin{definition}
If $G = G(m,n,k),$ $S$ is a base, and $x \in {S^*},$ then \emph{the orbit of}
$x$ in ${S^*}$ is the set
$orb(x,{S^*}) = \left\{ {xy:y \in I({S^*})} \right\}.$
\end{definition}

Since $S$ is a base, there are invertibles in ${S^*}.$
As noted earlier, $I({S^*})$ forms a group, thus $1 \in I({S^*})$ and it
follows that $x \in orb(x,{S^*}).$
If $G$ had non-trivial centre, there will be no invertibles with which to
create an orbit and a different approach will be required.

\begin{lemma}
\label{lemma:4.9}
If $G = G(m,n,k)$ and $S$ is a base, then, for each ${x_1},{x_2} \in {S^*},$
either $orb({x_1},{S^*}) = orb({x_2},{S^*})$ or
$orb({x_1},{S^*}) \cap orb({x_2},{S^*}) = \varnothing .$
\end{lemma}
\begin{proof}
Suppose that $orb({x_1},{S^*}) \cap orb({x_2},{S^*}) \ne \varnothing $ and
that $z \in orb({x_1},{S^*}) \cap orb({x_2},{S^*}).$
Thus there are ${y_1},{y_2} \in I({S^*})$ so that
${x_1}{y_1} = z = {x_2}{y_2}.$
It follows that ${x_1} = {x_2}{y_2}y_1^{ - 1}.$
An arbitrary element of $orb({x_1},{S^*})$ is of the form
${x_1}u\;(u \in I({S^*}),$ thus
${x_1}u = {x_2}({y_2}y_1^{ - 1}u) \in orb({x_2},{S^*}).$
It follows that $orb({x_1},{S^*}) \subseteq orb({x_2},{S^*}).$
The other containment is shown similarly and our result follows.
\end{proof}

The fact that ${S^*}$ is the union of its orbits together with
Lemma \ref{lemma:4.9} imply that the orbits of ${S^*}$ partition it into
equivalence classes with respect to the relation  defined as $x \sim y$ if and
only if there exists a $z \in I({S^*})$ for which $x = yz.$
In fact $\sim$ is a congruence; thus the quotient semigroup $S/\sim$ can be
formed.
The number of distinct orbits in ${S^*}$ is the order of the quotient
semigroup.
We will next show how these orbits are involved in the search for the
containers within ${\Sigma _G}(S).$

\begin{definition}
If $G = G(m,n,k),$ $S$ is a base, and $x \in {S^*},$ the orbit,
$orb(x,{S^*})$ is called \emph{basic} if $orb(x,{S^*})\cap S\ne\varnothing.$
\end{definition}

\begin{theorem}
\label{thm:4.10}
If $G = G(m,n,k),$ $S$ a base, and $x \in {S^*},$ then $orb(x,{S^*})$ is basic
if and only if ${\cal F}(x,S)$ is complete.
\end{theorem}
\begin{proof}
$\left(  \Rightarrow  \right)$
Since $orb(x,{S^*})$ is basic, we know there is an $s \in S$ and an invertible
$y \in I({S^*})$ for which $xy = s.$
Thus, representing the inverse of $y \pmod m$ as ${y^{ - 1}},$ we have
$x = s{y^{ - 1}}.$
Since $I({S^*})$ forms a group, ${y^{ - 1}} \in I({S^*}) \subseteq {S^*};$
therefore,
\[
\mu (x,{y^{ - 1}})
= \mu (s{y^{ - 1}},{y^{ - 1}})
\in {\Pi _\mu }(S) \subseteq {\Sigma _G}(S).
\]
Thus $C(x,{y^{ - 1}}) \subseteq {\Sigma _G}(S),$ by Lemma \ref{lemma:4.3}.
Since ${y^{ - 1}}$ is invertible, Lemma \ref{lemma:3.8}(iv) implies that
$C(x,{y^{ - 1}}) = C(x,1);$
therefore, $C(x,1) \subseteq {\Sigma _G}(S)$ and ${\cal F}(x,S)$ is complete.

$\left(  \Leftarrow  \right)$
If ${\cal F}(x,S)$ is complete, then $C(x,1) \subseteq {\Sigma _G}(S).$
Thus $\mu (x,1) \in {\Sigma _G}(S),$ by Lemma \ref{lemma:4.3}.
If $\mu (x,1) \in {\Gamma _\mu }(S),$ then $x \in S$ and, since
$x \in orb(x,{S^*}) \cap S,$$orb(x,{S^*})$ is basic.
If $\mu (x,1) \in {\Pi _\mu }(S),$ then $x = s{s^*},1 = {s^*}z$ and, since
${s^*}z = 1,$ ${s^*}$ is invertible.
Therefore, $x{({s^*})^{ - 1}} \in orb(x,{S^*}),$ and
$x{({s^*})^{ - 1}} = s{s^*}{({s^*})^{ - 1}} = s \in S.$
Thus $x{({s^*})^{ - 1}} \in orb(x,{S^*}) \cap S,$ and it follows that
$orb(x,{S^*})$ is basic.
\end{proof}

\begin{theorem}
\label{thm:4.11}
If $G = G(m,n,k)$ and $S$ is a base, then $orb(x,{S^*})$ is basic, for each
$x \in {S^*},$ if and only if  ${\Sigma _G}(S)$ is complete.
In this case we have
$\displaystyle\Sigma_G(S) = \dot{\bigcup}_{x\in S^*} C(x,1)$
and $\left| {{\Sigma _G}(S)} \right| = m\left| {{S^*}} \right|.$
\end{theorem}
\begin{proof}
By Theorem \ref{thm:4.10}, each orbit is basic if and only if each
$x$-family ${\cal F}(x,S)$ is complete.
This is the case if and only if ${\Sigma _G}(S)$ is complete.
The second sentence follows by Theorem \ref{thm:4.8}.
\end{proof}

\begin{corollary}
\label{cor:4.12}
If $G = G(m,n,k),$ $R = \{k_j: j\in\Z_n\}$ and $L = \{-k_j: j\in\Z_n\},$ then
\begin{enumerate}[(i)]
\item
If, for each $x \in {R^*},$ $orb(x,R)$ is basic, then
$\displaystyle\Sigma_G(S) = \dot{\bigcup}_{x\in R^*} C(x,1)$
and $\left| {{\rm P}(G)} \right| = \left| {{R^*}} \right|m,$ and
\item
If, for each $x \in {L^*},$ $orb(x,L)$ is basic, then
$\displaystyle\Sigma_G(S) = \dot{\bigcup}_{x\in L^*} C(x,1)$
and $\left| {\Lambda (G)} \right| = \left| {{L^*}} \right|m.$
\end{enumerate}
\end{corollary}
\begin{proof}
By Lemma \ref{lemma:4.2}, we know that ${\rm P}(G) = {\Sigma _G}(R)$ and $
\Lambda (G) = {\Sigma _G}(L).$
The result then follows immediately from Theorem \ref{thm:4.11}.
\end{proof}

If the orbit of $x$ is basic, then $ \cup {\cal F}(x,S) = C(x,1).$
However if the orbit of $x$ is non-basic, the containers in ${\cal F}(x,S)$
have a more complex interrelationship.
We now give examples to illustrate that, for $G(m,n,k)$ with trivial centre,
it is possible for each orbit to be basic in ${R^*}$ but not in ${L^*}$
and vice versa.
Thus the completeness of ${\rm P}(G)$ and $\Lambda (G)$ are independent.

\begin{example}
We leave the modular arithmetic calculations to the reader.
Note that $G(315,12,272)$ has trivial centre and that $orb(x,{R^*})$ is basic
for each $x \in {R^*},$ but $orb(225,{L^*})$ is not basic.
Also $G(135,12,62)$ has trivial centre  and $orb(x,{L^*})$ is basic for each
$x \in {L^*},$ but $orb(130,{R^*})$ is not basic.
In each case there is just one orbit which is not basic though, in general,
this is not the case.
A computer search shows that $63$ is the smallest value of $m$ for which
non-basic orbits exist in metacyclic groups with trivial centre for
${\rm P}(G)$ or $\Lambda (G).$
\end{example}

We will apply the following Lemma and Corollary to narrow the search for
non-basic orbits.

\begin{lemma}
\label{lemma:4.14}
If $G = G(m,n,k),$ $S$ is a base,  and $x \in I({S^*}),$ then $orb(x,{S^*})$
is basic.
\end{lemma}
\begin{proof}
Since $x \in I({S^*}),$ it is invertible.
Let us call the inverse ${x^{ - 1}}$ and note that, by Lemma \ref{lemma:4.5},
${x^{ - 1}} \in I({S^*}).$
Since $S$ is a base, we know that there is some $y \in I(S).$
Then
\[
y = 1\left( y \right) = (x{x^{ - 1}})y = x({x^{ - 1}}y) \in orb(x,{S^*}).
\]
Thus $y \in orb(x,{S^*}) \cap S$ and it follows that $orb(x,{S^*})$ is basic.
\end{proof}

\begin{corollary}
\label{cor:4.15}
Let $G = G(m,n,k)$ and let $S$ be a base.
If, for each $x \in N({S^*}) - N(S),$ $orb(x,S)$ is basic, then all orbits
are basic.
\end{corollary}
\begin{proof}
Let $x \in {S^*}.$
It is clear that if $x \in S,$ then $x \in S \cap orb(x,S),$ thus $orb(x,S)$
is basic.
By Lemma \ref{lemma:4.14}, if $x \in I({S^*}),$ then $orb(x,{S^*})$ is basic.
Therefore, if $orb(x,S)$ is non-basic, $x \in N({S^*})$ and $x \notin S.$
\end{proof}

\begin{example}
Let us return to the smallest non-abelian (metacyclic) group,
${S_3} = G(3,2,2).$
Since $({k_1},m) = (1,3) = 1,$ we know that ${S_3}$ has trivial centre.
We calculate the sets ${R^*}$ and ${L^*}$ as the multiplicative closures,
modulo 3, of
\[
R = \{k^j-1: j\in\Z_2\} = \{2^0 - 1, 2^1 - 1\} = \{0,1\}
\] and
\[
L = \{1-k^j: j\in\Z_2\} = \{1-2^0 , 1-2^1 \} = \{0,2\}
\]
(Note that $L =  - R.$)
Thus ${R^*} = \left\{ {0,1} \right\}$ and ${L^*} = \left\{ {0,1,2} \right\}.$
We must next verify that each orbit is basic.
By Corollary \ref{cor:4.15}, we need only check those $x \in N({S^*} - S).$
Since ${R^*} = R,$ this case requires no checking.
For ${L^*},$ we need only check to see if $orb(1,{L^*})$ intersects $L.$
This is true since
\[
orb(1,{L^*})
= \left\{ {1 \cdot y:y \in ({L^*})} \right\}
= \left\{ {1,2} \right\}
\]
and
\[
orb(1,{L^*}) \cap L
= \left\{ 2 \right\}
\ne \varnothing.
\]
Thus Corollary \ref{cor:4.12} applies and we conclude that
$\left| {{\rm P}({S_3})} \right| = \left| {{R^*}} \right|m = 2 \cdot 3 = 6$
and
$\left| {\Lambda ({S_3})} \right| = \left| {{L^*}} \right|m = 3 \cdot 3 = 9,$
as previously stated.
This is a kind of ``solution'' to the mystery of how these orders can be
different in the face of so much symmetry.
In fact we can identify the exact mappings contained in both
${\rm P}(G)$ and $\Lambda (G)$ using containers.
Theorem \ref{thm:4.11} implies that
\begin{align*}
{\rm P}({S_3})
= C(0,1) \cup C(1,1)
= \left\{ {\mu (0,0),\mu(0,1),\mu(0,2),\mu(1,0),\mu(1,1),\mu (1,2)}\right\}
\end{align*}
and
\begin{align*}
\Lambda ({S_3})
&= C(0,1) \cup C(1,1) \cup C(2,1)\\
&= \{ \mu (0,0),\mu(0,1),\mu(0,2),\mu(1,0),\mu(1,1),\mu(1,2),
      \mu(2,0),\mu(2,1),\mu (2,2)\} .
\end{align*}
\end{example}

\begin{example}
Since the construction of ${\Sigma _G}(S)$ may be of independent interest,
we will select a small $m$ for $G,$ with trivial centre, and choose a base
$S$ which will generate in ``interesting'' example of ${\Sigma _G}(S)$.
Let $G = G(5,4,3)$ and let $S = \left\{ {0,4} \right\}.$
Since $(k - 1,m) = (2,5) = 1,$ it follows that $G$ has trivial centre.
$S$ is a base since $0 \in N(S)$ and $4 \in I(S).$
We compute that ${S^*} = \left\{ {0,1,4} \right\}.$
To see that each orbit is basic, we need only check
$x \in {S^*} - S = \left\{ 1 \right\}.$
Since $orb(1,{S^*}) = \left\{ {1,4} \right\},$ we have
$orb(1,{S^*}) \cap S = \left\{ 4 \right\} \ne \varnothing .$
It follows that Theorem \ref{thm:4.11} and Corollary \ref{cor:4.12} hold in
this case.
Thus we see that ${\Sigma _G}(S)$ is a union of the maximal containers
$C(x,1)$ for $x \in {S^*}.$
It follows from our theorems, since $\left| {C(x,1)} \right| = 5$ and
$\left| {{S^*}} \right| = 3,$ that $\left| {{\Sigma _G}(S)} \right| = 15.$
In this case we may also compute that
${R^*} = {L^*} = \left\{ {0,1,2,3,4} \right\}$ and
$\left| {{\rm P}(G)} \right| = \left| {\Lambda (G)} \right| = 25;$
thus it is clear that ${\Sigma _G}(S)$ is a semigroup distinct from the
commutation semigroups.
\end{example}

\section{Non-basic orbits}
\label{sec:5}

In the previous section we have seen that, given $G(m,n,k)$ and a base $S$,
if each orbit is basic, ${\Sigma _G}(S)$ is complete.
Theorem \ref{thm:4.8} then gives us a description of the mu-maps in
${\Sigma _G}(S)$ in terms of containers as well as an easily calculated
formula for its order.
The other case to consider is the occurrence of non-basic orbits in ${S^*}.$
Here we have a more complex situation for which a uniform description of the
mu-maps in ${\Sigma _G}(S)$ is more difficult to obtain.
Thus we will provide a procedure which lists the $x$'s with non-basic
orbits.
In each case we will then apply Lemma \ref{lemma:4.5} to generate
${{\cal F}_G}(x,S).$
Having done this for each non-basic $x,$ we can then find exactly those
containers which constitute ${\Sigma _G}(S).$
When $orb(x,{S^*})$ is non-basic, we may need to include more than one
$x$-container in the union.
Some $x$-families have one container which is a superset of all other
members of the family, and this can be used as the only $x$-container
in the union.
However, some $x$-families require the union of several containers.
These containers will not be disjoint; thus, to determine the order of the
$x$-family portion of the union forming ${\Sigma _G}(S),$ we may have to use
the principle of inclusion and exclusion.
We will then work through an example to illustrate the procedure in action.

\textbf{Procedure.}
We assume we are given $G = G(m,n,k),$ a metacyclic group with trivial centre,
and a base $S\subseteq \Z_m.$
We generate ${S^*}$ by closing $S$ under multiplication and then write
$S^* = I(S^*) \dot{\bigcup} N(S^*).$

Next, we calculate the orbits, for each $x \in {S^*},$ by forming the sets
$orb(x,{S^*}) = \left\{ {xy:y \in I({S^*})} \right\}.$

By Corollary \ref{cor:4.15}, we need only check the orbits for
$x \in N({S^*}) - N(S)$ to see if $orb(x,{S^*}) \cap S = \varnothing .$
If each orbit intersects $S,$ then Theorem \ref{thm:4.11} tells us that
${\Sigma _G}(S)$ is complete, how to write it as a union of containers, and
its order.

When an orbit does not intersect $S,$ we add $orb(x,{S^*})$ to the list of
non-basic orbits.
Assuming there are non-basic orbits, write ${S^*} = B \cup N,$ with
$B = \left\{ {x \in {S^*}:orb(x,{S^*}){\rm{ is basic}}} \right\}$ and
$N = \left\{ {x \in {S^*}:orb(x,{S^*}){\rm{ is not basic}}} \right\}.$
By Theorem \ref{thm:4.10}, we can write the portion of ${\Sigma _G}(S)$
covered by families of containers associated with basic orbits, as with size
$m\left| B \right|.$
The remaining portion of ${\Sigma _G}(S)$ is
$\displaystyle\dot{\bigcup}_{x\in N} \cup {\cal F}(x,S).$
We know that, for each $x \in N,$
$ \cup {\cal F}(x,S) = \bigcup\limits_{y \in Y(x)} {C(x,y)} .$
The first step in determining these unions is to calculate the set $Y(x)$ for
a particular $x \in N.$
Since $N \cap S = \varnothing ,$ we know that $\mu (x,y) \in {\Pi _\mu }(S).$
Therefore, there exist $s \in S$ and ${s^*} \in {S^*}$ such that $x = s{s^*}$
and $y = {s^*}z.$
We find all such pairs $({s^*},z)$ and write the list of containers
$C(x,{s^*}z)$ that the pairs yield.
This set is the $x$-family.
The union of these containers is the portion of  ${\Sigma _G}(S)$ contributed
to the union by $\cup {\cal F}(x,S).$
Looking at a list of such containers, we need to determine the containment
relationships among them.
We can make use of Lemma \ref{lemma:3.8} and the principle of inclusion and
exclusion to calculate the size of this union.

At this point an example will be useful.

\begin{example}
Taking $G = G(63,6,2),$ we see that since $k - 1 = 1,$ it is coprime to $63$
and, therefore, $G$ has trivial centre.
We will construct ${\rm P}(G)$ and determine its order.
Here a calculation gives
\[
R = \left\{ {0,\underline 1 ,3,7,15,\underline {31} } \right\}
\]
and, closing
this under multiplication, we have
\[
{R^*} = \{ 0,\underline 1 ,3,\underline 4 ,6,7,9,12,15,\underline {16} ,18,
21,24,27,28,30,\underline {31} ,33,36,39,42,45,48,49,51,54,\underline {55} ,
57,60,\underline {61} \},
\]
where the invertible elements in $R$ and ${R^*}$ have been underlined.
The values of $x$ for which we wish to check the orbits are in
\[
N({R^*}) - N(R)
= \{ 6,9,12,18,21,24,27,28,30,33,36,39,42,45,48,49,51,54,57,60\}.
\]
Note that $31$ generates the group of invertibles $I({R^*});$
therefore we can multiply repeatedly by $31$ to produce each orbit.
When we do this, we find three non-basic orbits:
$or{b_R}(9),$ $or{b_R}(21),$ and $or{b_R}(42).$
So for each value of $x \in {R^*} - \left\{ {9,21,42} \right\},$ the
$x$-family is complete and its contribution to ${\rm P}(G)$ is $C(x,1).$
Each of these maximal containers has order $63.$

Next consider the $9$-family.
Here we wish to find all solutions of the congruence
$uv = 9\;(\bmod 63)$ for $u \in R$ and $v \in N({R^*}).$
The modular arithmetic here may be simplified by noting that if we write $uv$
as $9u'v',$ we can reduce $9u'v' = 9\;(\bmod 9 \cdot 7)$ to
$u'v' = 1\;(\bmod 7),$ yielding
\[
\left\{ {u',v'} \right\}
\in \left\{ {\{ 1,1\} \left\{ {2,4} \right\},\left\{ {3,5} \right\},
    \left\{ {6,6} \right\}} \right\},
\]
using multisets, since we can ignore order temporarily, for these pairs.
Thus, depending on how the two factors of $3$ are distributed between $u$ and
$v,$
\[\left\{ {u,v} \right\}
\in \{ \{ 1,9\} ,\{ 3,3\} ,\{ 18,4\} ,\{ 6,12\} ,\{ 2,36\} ,
\{ 27,5\} ,\{ 9,15\} ,\{ 3,45\} ,\{ 54,6\} ,\{ 18,18\} \}.
\]
Since $u \in R,$ we can remove any doubleton with no coordinate in $R.$
This leaves us with
\[
\{ \{ 1,9\} ,\{ 3,3\} ,\{ 9,15\} ,\{ 3,45\} \} .
\]
Checking that $v \in N({R^*}),$ we see that all four doubletons are solutions.
Thus we have
\[
\left( {u,v} \right) \in \{ (1,9),(3,3),(3,45),(15,9)\} ;
\]
hence,
\[
{{\cal F}_G}(9,R) = \{ C(9,9),C(9,3),C(9,45)\} .
\]
By Lemma \ref{lemma:3.8} and \ref{lemma:3.9}, we have
$C(9,45) = C(9,9) \subseteq C(9,3)$ and therefore,
$ \cup {{\cal F}_G}(9,R) = C(9,3)$ and
$\left| { \cup {{\cal F}_G}(9,R)} \right| = \left| {C(9,3)} \right| = 21.$

To find $ \cup {{\cal F}_G}(21,R),$ we solve the congruence
$uv = 21\;(\bmod 63)\;(u \in R,v \in N({R^*}))$ by removing $21$ to obtain
$u'v' = 1\;(\bmod 3).$
The two solutions, $(1,1)$ and $(2,2),$ yield the doubletons
$\{ \{ 1,21\} ,\{ 3,7\} ,$$\{ 2,42\} , \{ 6,14\} \} .$
Checking the domains, we obtain the solutions of the original congruence,
$\{ (1,21),(3,7),(7,3)\} .$
Thus ${{\cal F}_G}(21,R) = \left\{ {C(21,21),C(21,7),C(21,3)} \right\}.$
Note, by Lemma \ref{lemma:3.9}, that $C(21,21) \subseteq C(21,3)$ and
$C(21,21) \subseteq C(21,7),$ but $C(21,3)$  and $C(21,7)$ are incomparable.
Also $C(21,21) = C(21,3) \cap C(21,7);$
therefore, by the law of inclusion and exclusion,
\[
\left| {{{\cal F}_G}(21,R)} \right|
= \left| {C(21,3)} \right| + \left| {C(21,7)}\right|-\left|{C(21,21)}\right|
= 21 + 9 - 3 = 27.
\]
Similarly, we find that $ \cup {{\cal F}_G}(42,R) = C(42,3)$ with
$\left| { \cup {{\cal F}_G}(42,R)} \right| = 21.$
In summary, there are $27$ elements of ${R^*}$ with basic orbits, therefore
that portion of ${\rm P}(G)$ is
$\left( {\bigcup\limits_{orb(x,{R^*})basic} {C(x,1)} } \right)$ having order
$27 \cdot 63 = 1701.$
The remainder of ${\rm P}(G)$ consists of the unions of the three families
calculated.
Thus
\begin{align*}
{\rm P}(G)
&= \left( {\bigcup\limits_{orb(x,{R^*})basic} {C(x,1)} } \right) \cup
   \left( { \cup {{\cal F}_G}(9,R)} \right) \cup
   \left( { \cup {{\cal F}_G}(21,R)} \right) \cup
   \left( { \cup {{\cal F}_G}(42,R)} \right)\\
&= \left( {\bigcup\limits_{orb(x,{R^*})basic} {C(x,1)} } \right) \cup
   C(9,3) \cup ((C(21,3) \cup C(21,7)) - C(21,21)) \cup C(42,3)
\end{align*}
with $\left| {{\rm P}(G)} \right| = 1701 + 21 + 27 + 21 = 1770.$
\end{example}

\section{Application of the general theory to the commutation semigroups}
\label{sec:6}

In this section we will apply the general results obtained in Section
\ref{sec:4} to the commutation semigroups ${\rm P}(G)$ and $\Lambda (G)$ for
$G = G(m,n,k)$ a metacyclic group with trivial centre.
Specifically we will investigate situations in which $m$ and $n$ are of,
number theoretically, simple form.
As mentioned earlier, $m = 63$ is the first value for which non-basic orbits
exist.
Note that $63( = {3^2} \cdot 7)$ is of the form ${p^2}q.$
We will show, in this section, that if $m$ is of the form $p$ or ${p^2}$ or if
$n$ is prime, then there are no non-basic orbits and the commutation
semigroups are complete.
Since this can fail when $m = {p^2}q,$ it would be interesting to study the
situation for $m = pq$ and ${p^3}.$

\begin{theorem}
\label{thm:6.1}
If $G = G(p,n,k)$ with $p$ prime and $S$ is a base, then ${\Sigma _G}(S)$
is complete.
\end{theorem}
\begin{proof}
Since $p$ is prime, all non-zero elements of  $\Z_p$ are invertible.
Thus $N({S^*}) = \left\{ 0 \right\}.$
Since $0 \in S,$ we have  $N({S^*}) - N(S) = \varnothing ,$ and thus,
by Corollary \ref{cor:4.15}, it follows that $orb(x,R)$ is basic for each
$x \in {S^*}.$
The result follows by Theorem \ref{thm:4.11}.
\end{proof}

Theorem \ref{thm:6.1} with Lemma \ref{lemma:4.2} imply the following.

\begin{corollary}
\label{cor:6.2}
If $G = G(p,n,k)$ with $p$ prime, then ${\rm P}(G)$ and $\Lambda (G)$
are complete.\qed
\end{corollary}

\begin{theorem}
\label{thm:6.3}
If $G = G({p^2},n,k)$ with $p$ prime, then ${\rm P}(G)$ and $\Lambda (G)$
are complete.
\end{theorem}
\begin{proof}
We will prove the result for ${\rm P}(G)$ and comment that the proof for
$\Lambda (G)$ is similar.
Given $m = {p^2},$ for some prime $p,$ we claim that either
$N(R) = \left\{ 0 \right\}$ or
$N(R) = \left\{ {0,p,2p, \ldots ,(p - 1)p} \right\}.$
Assuming this has been shown, Corollary \ref{cor:4.15} says it is enough to
check that the elements of $N({R^*}) - N(R)$ generate basic orbits.
If $N(R) = \left\{ 0 \right\},$ it is clear that the only non-invertible in
the closure of $R$ would be $0$ itself.
In this case, $N({R^*} - R) = \varnothing ,$ and Corollary \ref{cor:4.15}
implies that all orbits are basic.
Hence, by Theorem \ref{thm:4.11}, ${\rm P}(G)$ is complete.
If $N(R) = \left\{ {0,p,2p, \ldots ,(p - 1)p} \right\},$ this is the complete
set of non-invertibles in $\Z_{p^2}$ and therefore no new non-invertibles
could be generated in ${R^*}$ when forming their products.
Again, Corollary \ref{cor:4.15} assures us that all orbits are basic and
Theorem \ref{thm:4.11} yields our conclusion.

It remains to prove the following:

\emph{Claim.}
Either $N(R) = \left\{ 0 \right\}$ or
$N(R) = \left\{ {0,p,2p, \ldots ,(p - 1)p} \right\}.$

\emph{Proof of Claim.}
We will suppose that $N(R) \ne \left\{ 0 \right\}$ and show that
$N(R) = \left\{ {0,p,2p, \ldots ,(p - 1)p} \right\}.$
Note that each of the elements of the form $ap\;(0 \le a \le p - 1)$ has a
common factor of $p$ with $m( = {p^2})$ and, hence, is non-invertible in
$\Z_{p^2}.$
We are assuming there is a non-zero, non-invertible in $R,$ thus there exists
an $a\;(0 < a \le p - 1)$ and a $t\;(1 \le t \le n - 1)$ so that
${k^t} - 1 = ap\;\pmod {p^2}.$
Note that $a$ is invertible modulo $p;$ thus there is an
$s\;(1 \le s \le p - 1)$ so that $as = 1(\bmod p).$
Thus there is $u\;(1 \le u \le p - 1)$ for which $sa = 1 + up.$
If $a,$ $s,$ and $u$ are interpreted as integers modulo ${p^2},$ we have
$sa = 1 + up\;\pmod {p^2}.$
Since ${k^t} - 1 = ap\;\pmod {p^2},$ we have ${k^t} = ap + 1\;\pmod {p^2},$
and thus, ${k^{st}} = {(ap + 1)^s}\pmod {p^2}.$
By the binomial theorem,
\[
{k^{st}} = {(ap + 1)^s}\pmod {p^2} = \left( {\begin{array}{*{20}{c}}
s\\
0
\end{array}} \right){a^s}{p^s} + \left( {\begin{array}{*{20}{c}}
s\\
1
\end{array}} \right){a^{s - 1}}{p^{s - 1}} +  \cdots  + \left(
{\begin{array}{*{20}{c}}
s\\
{s - 1}
\end{array}} \right)ap + 1\;\pmod {p^2}.
\]
Reducing these terms modulo ${p^2},$ we obtain
\[
{k^{st}}
= sap + 1\;\pmod {p^2} = (1 + up)p + 1\;\pmod {p^2}
= p + 1\;\pmod {p^2}.
\]
Therefore, ${k^{st}} - 1 = p,$ and it follows that $p \in R.$
If $b\in\Z_{p^2}\cap\{1,2,\ldots,p-1\},$ we note that
${k^{bst}} = {(p + 1)^b}\;\pmod {p^2}.$
Again, by the binomial theorem we have
\[
{k^{bst}} = {(p + 1)^b}\;\pmod {p^2} = bp + 1\;\pmod {p^2}
\]
and, hence, $bp \in N(R).$
This establishes our claim.
\end{proof}

Next we introduce a technical lemma.
Recall that ${k_t} = {k^t} - 1,$ that $k$ and ${k_1}$ are both coprime to $m,$
that $n = in{d_m}(k)$ and, hence, that ${k_n} = 0 \pmod m.$

\begin{lemma}
\label{lemma:6.4}
Let $G = G(m,n,k)$ and let $p$ be a prime which divides $m.$
Let $s\;\left( {1 < s \le n} \right)$ be minimal so that $p$ divides ${k_s}.$
Then, for each $t\;\left( {1 < t \le n} \right),$ $p$ divides ${k_t}$ if and
only if $s$ divides $t.$
\end{lemma}
\begin{proof}
We will first justify the existence of the number $s$ in the statement of the
Lemma.
Note that since ${k_n} = 0 \pmod m,$ we know that $p$ divides ${k_{n.}}$
Thus there is a minimal $s$ $(s \le n)$ for which $p$ divides ${k_s}.$
Since $(m,{k_1}) = 1,$ $p$ does not divide ${k_1}$ and, therefore,
$1 < s \le n.$

$\left(  \Rightarrow  \right)$
Since  ${k_s} = {k_1}(1 + k + \ldots + {k^{s - 1}})$ and
${k_t} = {k_1}(1 + k + \ldots + {k^{t - 1}}),$ and since $p$ does not divide
${k_1},$ it divides both $(1 + k + \ldots + {k^{s - 1}})$ and
$(1 + k + \ldots + {k^{t - 1}}).$
Let $q$ and $r$ be the non-negative integers so that $t = qs + r$ with
$0 \le r < s$.
If $r = 0,$ then $s$ divides $t$ and we are done.
Suppose then that $r > 0.$ Then
\begin{align*}
(1 + k + \ldots + {k^{t - 1}})
&= (1 + k + \ldots + {k^{s - 1}}) + ({k^s} + {k^{s + 1}} + \ldots +\\
&\hspace*{2pc} {k^{2s - 1}}) + ({k^{2s}} + {k^{2s + 1}} +
    \ldots +  {k^{3s - 1}}) + \ldots +\\
&\hspace*{2pc} ({k^{(q - 1)s}} + {k^{(q - 1)s + 1}} +
  \ldots + {k^{qs - 1}}) + ({k^{qs}} + {k^{qs + 1}} +
  \ldots + {k^{qs + r - 1}}) \\
&= (1 + k + \ldots + {k^{s - 1}}) + {k^s}(1 + k +
  \ldots + {k^{s - 1}}) + \ldots + \\
&\hspace*{2pc} {k^{(q - 1)s}}(1 + k + \ldots + {k^{s - 1}}) + {k^{qs}}(1 + k +
  \ldots + {k^{r - 1}})\\
&= (1 + {k^s} + {k^{2s}} + \ldots + {k^{(q - 1)s}})(1 + k +
  \ldots + {k^{s - 1}}) + {k^{qs}}(1 + k + \ldots + {k^{r - 1}}).
\end{align*}
Since $p$ divides $(1 + k + \ldots + {k^{s - 1}})$ and
$(1 + k + \ldots + {k^{t - 1}}),$ it follows that $p$ must also divide
${k^{qs}}(1 + k + \ldots + {k^{r - 1}}).$
By Lemma \ref{lemma:2.1}, $p$ does not divide $k;$ therefore it must divide
$(1 + k + \ldots + {k^{r - 1}})$ and, hence, $p$ divides
${k_r} = (k - 1)(1 + k + \ldots + {k^{r - 1}}).$
But since $r < s,$ this contradicts the minimality of $s.$ Thus $r = 0,$
and our result follows.

$\left(  \Leftarrow  \right)$
Now suppose that $s$ divides $t\;(1 < t \le n)$ with $t = qs$ for some
positive integer $q.$
From a calculation similar to the one above, we derive
\[
(1 + k + \ldots + {k^{t - 1}})
= (1 + k + \ldots + {k^{qs - 1}})
= (1 + {k^s} + {k^{2s}} + \ldots + {k^{(q - 1)s}})(1 + k +
  \ldots + {k^{s - 1}}).
\]
We next multiply both sides by ${k_1}:$
\[
{k_1}(1 + k + \ldots + {k^{t - 1}})
= (1 + {k^s} + {k^{2s}} + \ldots + {k^{(q - 1)s}}){k_1}(1 + k +
  \ldots + {k^{s - 1}}).
\]
Therefore, ${k_t} = (1 + {k^s} + {k^{2s}} + \ldots + {k^{(q - 1)s}}){k_s},$
and since $p$ divides ${k_s},$ it follows that $p$ divides ${k_i}.$
\end{proof}

\begin{theorem}
\label{thm:6.5}
If $G = G(m,p,k)$ with $p$ prime, then
$R^* - \{0\}, L^*-\{0\} \subseteq I(\Z_m)$ and both ${\rm P}(G)$ and
$\Lambda (G)$ are complete.
\end{theorem}

\begin{proof}
Suppose  $p$ is a prime which divides $m.$
Since ${k_n} = 0 \pmod m,$ we know that ${k_n}$ must have $p$ as a divisor.
Select $s$ minimal so that $p$ divides ${k_s}.$
We know that ${k_1} = k - 1$ is coprime to $m,$ thus $p$ does not divide
${k_1}.$
So $1 < s \le n.$
By Lemma \ref{lemma:6.4}, it follows that $s$ divides $n,$ but, since $n$ is
prime, we have $s = n.$
It follows that $p$ is not a divisor of any ${k_i}$ with $i < n.$
This argument applies to each prime dividing $m;$ therefore we can conclude
that no prime divisor of $m$ divides ${k_i}$ with $i < n.$
Thus all such ${k_i}$ are coprime to $m$ and, hence they, and their products,
are invertible in $\Z_m.$
By Lemma \ref{lemma:2.4}(i), $0 \in R,$ thus 0 is the only non-invertible in
${R^*},$ and it follows that $N({R^*}) - N(R) = \varnothing .$
Then, by Corollary \ref{cor:4.15}, all orbits are basic and our result for
${\rm P}(G)$ follows from Theorem \ref{thm:4.11}.
Note that if $x$ is invertible in $\Z_m$ with inverse ${x^{ - 1}}$, then
$( - x)( - {x^{ - 1}}) = x{x^{ - 1}} = 1.$
Thus all non-zero elements of ${L^*}$ are also invertible.
So both the left are right commutation semigroups are complete.
\end{proof}

\begin{theorem}
\label{thm:6.6}
Any non-abelian $pq$-group is a metacyclic group with trivial centre
and its commutation semigroups are complete.
\end{theorem}
\begin{proof}
We assume, without loss of generality, that $p > q.$
It is easily seen that a non-abelian group of order $pq$ has presentation
$G(p,q,k).$
Since $p$ is prime, $(p,k - 1) = 1;$ thus $G$ has trivial centre.
Since $m = p$ is prime, Theorem \ref{thm:6.1} gives our result.
\end{proof}

This result applies to the $pq$-groups studied by Countryman in
\cite{countryman1970}.
Thus each of the commutation semigroups of a non-abelian $pq$-group
is simply a disjoint union of maximal containers.
Each maximal container is of order $p$ and the order of the two commutation
semigroups are determined by the sizes of the multiplicative closures of $R$
and $L$ in $\Z_p.$
For example, in the $pq$-group $G(7,2,6)$ we have
$R = \left\{ {0,5} \right\},$ ${R^*} = \left\{ {0,1,2,3,4,5,6} \right\},$
$L = \left\{ {0,2} \right\},$ and ${L^*} = \left\{ {0,1,2,4,} \right\}.$
Thus, by Corollary \ref{cor:4.12}, we have
${\rm P}(G) = \bigcup\limits_{x \in {R^*}} {C(x,1)} $ and
$\Lambda (G) = \bigcup\limits_{x \in {L^*}} {C(x,1)} $ with
$\left| {{\rm P}(G)} \right| = \left| {{R^*}} \right|7 = 49,$ and
$\left| {\Lambda (G)} \right| = \left| {{L^*}} \right|7 = 28.$

In \cite{countryman1970} (Theorem 2.1) Countryman proves:
If $G$ is a non-abelian $pq$-group $(p,$ $q$ primes), then
${\rm P}(G) = \Lambda (G)$ if and only if
$\left| {{\rm P}(G)} \right| = \left| {\Lambda (G)} \right|$.
He also notes that these two conditions are equivalent to
${\rm P}(G) \cong \Lambda (G).$
Having developed the theory to this point, we are now able to extend his
result.

\begin{theorem}
\label{thm:6.7}
If $G = G(p,n,k)$ with $p$ a prime, then the following are equivalent:
\begin{enumerate}[(i)]
\item ${\rm P}(G) = \Lambda (G),$
\item ${\rm P}(G) \cong \Lambda (G),$
\item $\left| {{\rm P}(G)} \right| = \left| {\Lambda (G)} \right|,$
\item $\left| {{R^*}} \right| = \left| {{L^*}} \right|.$
\end{enumerate}
\end{theorem}

\begin{proof}
First note that, by Theorem \ref{thm:6.5}, all non-zero elements of
${R^*}$ and ${L^*}$ are invertible.

$(i) \Rightarrow (ii)$ and $(ii) \Rightarrow (iii)$ are clear.

$(iii) \Rightarrow (iv):$
By Corollary \ref{cor:6.2}, we know that ${\rm P}(G)$ and $\Lambda (G)$ are
complete. Thus, Corollary \ref{cor:4.12}(i) and (ii) imply that
$\left| {{\rm P}(G)} \right| = \left| {{R^*}} \right|p$ and
$\left| {\Lambda (G)} \right| = \left| {{L^*}} \right|p.$
By hypothesis (iii), this yields
$\left| {{R^*}} \right| = \left| {{L^*}} \right|.$

$(iv) \Rightarrow (i):$
Note that ${R^*} - \left\{ 0 \right\}$ and ${L^*} - \left\{ 0 \right\}$ are
subgroups of $\Z_p-\{0\}$ of the same order.
Since $\Z_p-\{0\}$ is the multiplicative group of a finite field, it is
cyclic, and since both ${R^*} - \left\{ 0 \right\}$ and
${L^*} - \left\{ 0 \right\}$ are subgroups of a cyclic group, they are cyclic.
Since cyclic groups have only one subgroup of each possible order, we conclude
that ${R^*} - \left\{ 0 \right\} = {L^*} - \left\{ 0 \right\}.$
Thus ${R^*} = {L^*}.$
By Corollary \ref{cor:4.12}(i) and (ii) we have
$\mathrm{P}(G) = \displaystyle\dot{\bigcup}_{x\in R^*} C(x,1)$ and
$\Lambda(G) = \displaystyle\dot{\bigcup}_{x\in L^*} C(x,1);$
therefore,
${\rm P}(G) = \Lambda (G).$
\end{proof}

\end{document}